\documentclass[twoside,11pt]{article}

\usepackage[preprint, nohyperref]{jmlr2e}

\usepackage[utf8]{inputenc} 
\usepackage[T1]{fontenc}    
\usepackage{booktabs}       
\usepackage{amsfonts}       
\usepackage{nicefrac}       
\usepackage{microtype}      
\usepackage{xcolor}         

\usepackage{mathtools}
\usepackage{xspace}
\usepackage{enumitem}
\usepackage{nicefrac}

\usepackage{algorithm}
\usepackage[noend]{algorithmic}

\usepackage[colorlinks=true,citecolor=blue,urlcolor=black]{hyperref} 

\usepackage[capitalise]{cleveref}

\newtheorem{assumption}{Assumption}
\newtheorem{question}{Question}

\newcommand{\R}{\mathbb{R}}

\newcommand{\fedavg}{\textsc{FedAvg}\xspace}
\newcommand{\fedavgm}{\textsc{FedAvgM}\xspace}
\newcommand{\fedadagrad}{\textsc{FedAdagrad}\xspace}
\newcommand{\serveropt}{\textsc{ServerOpt}\xspace}

\DeclarePairedDelimiterX{\abs}[1]{\lvert}{\rvert}{#1}
\DeclarePairedDelimiterX{\norm}[1]{\lVert}{\rVert}{#1}

\ShortHeadings{Iterated Vector Fields, Conservatism, and Federated Learning}{Charles and Rush}
\firstpageno{1}

\begin{document}

\title{Iterated Vector Fields and Conservatism, with Applications to Federated Learning}

\author{\name{Zachary Charles} \email{zachcharles@google.com} \\ \addr Google Research \AND
 \name{Keith Rush} \email{krush@google.com}\\
 \addr Google Research}
 
\editor{}

\maketitle

\begin{abstract}
We study whether iterated vector fields (vector fields composed with themselves) are conservative. We give explicit examples of vector fields for which this self-composition preserves conservatism. Notably, this includes gradient vector fields of loss functions associated with some generalized linear models. As we show, characterizing the set of vector fields satisfying this condition leads to non-trivial geometric questions. In the context of federated learning, we show that when clients have loss functions whose gradients satisfy this condition, federated averaging is equivalent to gradient descent on a surrogate loss function. We leverage this to derive novel convergence results for federated learning. By contrast, we demonstrate that when the client losses violate this property, federated averaging can yield behavior which is fundamentally distinct from centralized optimization. Finally, we discuss theoretical and practical questions our analytical framework raises for federated learning.
\end{abstract}

\section{Introduction}

In this work, we consider vector fields of the form $V: \R^n \to \R^n$. Recall that $V$ is conservative if there is some differentiable function $f: \R^n \to \R$ such that $V = \nabla f$. We are interested in whether \emph{iterated} vector fields (vector fields of the form $V \circ V \circ \dots \circ V$) are conservative. While mathematically rich in its own right, this question has important connections to dynamical systems and optimization. As we will show, conservative iterated vector fields are particularly important for understanding optimization algorithms for federated learning.

\paragraph{Notation.} Let $\mathcal{V}(\R^n, \R^m)$ be the collection of functions from $\R^n$ to $\R^m$. We let $\mathcal{C}^k(\R^n, \R^m)$ denote the subset of $\mathcal{V}(\R^n, \R^m)$ of functions whose coordinate functions are all of class $\mathcal{C}^k$. If $m = n$, we abbreviate these by $\mathcal{V}(\R^n)$ and $\mathcal{C}^k(\R^n)$. Throughout, $\norm{\cdot}$ denotes the $\ell_2$ norm on $\R^n$ with corresponding inner product $\langle \cdot, \cdot \rangle$, and $I \in \mathcal{V}(\R^n)$ denotes the identity map.

Given $V \in \mathcal{V}(\R^n)$, we use exponents to denote repeated iterations of $V$. That is, for $k \geq 1$ we define:
\[
V^k(x) := \underbrace{V\circ V \circ \dots \circ V}_{\text{k times}}(x)
\]
By convention, for any $V \in \mathcal{V}(\R^n)$ we define $V^0 := I$.

\paragraph{Summary.} Let $V \in \mathcal{V}(\R^n)$, and $k$ be a positive integer. We study the following question.
\begin{question}\label{question1}
If $V$ is conservative, is $V^k$ also conservative?
\end{question}

This leads to the following definition.
\begin{definition}
$V$ is $k$-conservative if $V^k$ is conservative. $V$ is $\infty$-conservative if $V^k$ is conservative for all $k \geq 1$.
\end{definition}

For convenience, we use ``conservative'' and ``$1$-conservative'' interchangeably. In a slight abuse of notation, we say that $\mathcal{A}\subseteq\mathcal{V}(\R^n)$ is $k$-conservative if for all $V \in \mathcal{A}$, $V$ is $k$-conservative. In order to show that $\mathcal{A}$ is $\infty$-conservative, it suffices to show that $\mathcal{A}$ is conservative and closed under self-composition, as reflected in the following definition.
\begin{definition}
$\mathcal{A}\subseteq \mathcal{V}(\R^n)$ is closed under self-composition if for all $V \in \mathcal{A}$ and $k\geq 1,$ $V^k \in \mathcal{A}$.
\end{definition}
This leads us to the following specialization of \cref{question1}.
\begin{question}\label{question2}
Let $\mathcal{A}\subseteq\mathcal{V}(\R^n)$ be conservative. Is $\mathcal{A}$ closed under self-composition?
\end{question}

\paragraph{Vector Fields and Optimization.} Motivated by optimization, we will often consider vector fields of the form $V(x) = \nabla f(x)$, where $f: \R^n \to \R$ is differentiable. Given a set $\mathcal{F}$ of differentiable functions mapping $\R^n$ to $\R$, we define $\nabla \mathcal{F} = \{V \in \mathcal{V}(\R^n) : V = \nabla f, f \in \mathcal{F}\}$. For $\gamma \in \R$, we define $I - \gamma\nabla\mathcal{F} := \{I - \gamma\nabla f : f \in \mathcal{F}\}$.
A recurring theme in this work is whether a set $I - \gamma\nabla \mathcal{F}$ is $k$-conservative. Such vector fields arise naturally in optimization, as gradient descent on a function $f$ with learning rate $\gamma$ corresponds to the discrete-time dynamical system given by $x_{t+1} = (I-\gamma\nabla f)(x_t)$.

Given an initial point $x_0$, the iterates of gradient descent then satisfy $x_k = V^k(x_0)$ where $V = I-\gamma\nabla f$. Therefore, if $I-\gamma\nabla f$ is $\infty$-conservative, then the $k$-th iterate of gradient descent is actually $\nabla h_k(x_0)$ for some function $h_k: \R^n \to \R$. While this observation may not shed light on centralized optimization, it will prove much more useful when trying to understand the behavior of federated optimization algorithms, as we discuss below.

\section{Connections to Federated Learning}\label{sec:federated_learning}

In federated learning, we often have clients $c = 1, 2, \dots, C$, each with a differentiable loss function $f_c: \R^n \to \R$. The clients can all communicate with some shared server. In many settings, the server would like to minimize the average of the client loss functions:

\begin{equation}
    \min_x f_{avg}(x) := \dfrac{1}{C}\sum_{c=1}^C f_c(x).
\end{equation}
One noteworthy approach to federated learning is \emph{federated averaging} (\fedavg) \citep{mcmahan2017communication}. In this work, we analyze a somewhat simplified, deterministic version of \fedavg (sometimes referred to as local gradient descent~\citep{khaled2019first}) in which all clients participate in every round, and each client uses gradient descent to perform local optimization. In detail, this simplified \fedavg  operates as follows.

The server maintains some global model and uses multiple rounds of communication with the clients to update this model. At each round of \fedavg, the server broadcasts its model to the clients. The clients perform $k$ steps of gradient descent (with learning rate $\gamma$) on their respective loss functions, and send the resulting models to the server. The server then updates its model as the average of these client models, and repeats this process. A full description of this method is given in Algorithm \ref{alg:fedavg}.

\setlength{\textfloatsep}{10pt}
\begin{algorithm}[ht]
    \begin{algorithmic}
	\caption{Simplified \fedavg (aka Local Gradient Descent)}
	\label{alg:fedavg}
	\STATE {\bf Input:} Client loss functions $\{f_c\}_{c=1}^C, k \geq 1, \gamma > 0$, $T \geq 1$, initial model $x_0$,
	\FOR  {$t=0, \cdots, T-1$}
		\STATE The server broadcasts its model $x_{t}$ to all clients.
		\STATE Each client $c$ performs $k$ steps of gradient descent on $f_c$ with step-size $\gamma$ starting at $x_t$.
		\STATE After training, each client $c$ sends its local model $x_t^c$ to the server.
		\STATE The server updates its model via $x_{t+1} = C^{-1}\sum_{c=1}^C x_t^c$.
    \ENDFOR
	\end{algorithmic}
\end{algorithm}

Since communication from clients to the server is frequently a bottleneck~\citep{mcmahan2017communication, bonawitz2019towards}, this algorithm is often practical only when $k > 1$. When $k = 1$, this is equivalent (from the perspective of the server models $\{x_t\}_{t=0}^T$) to gradient descent with learning rate $\gamma$ on $f_{avg}$, the average of the client loss functions.

We now rephrase \cref{alg:fedavg} in terms of iterated vector fields. Define $V_c := I-\gamma\nabla f_c$. At each round $t$ of \fedavg, each client computes $V^k_c(x_t)$, and the server updates its model via the discrete-time dynamical system $x_{t+1} = C^{-1}\sum_{c=1}^C V^k_c(x_t)$. This ``operator-theoretic'' view of \fedavg has been previously used to leverage techniques from operator theory to analyze and design federated learning algorithms \citep{malinovskiy2020local, pathak2020fedsplit, malekmohammadi2021operator}.

In order to allow more general ``server optimization'' in \fedavg, \citet{reddi2021adaptive} propose a ``model delta'' variant. In our setting, this corresponds to the server update
\begin{equation}\label{eq:fedavg_update}
    x_{t+1} = x_t - \dfrac{\eta}{C}\sum_{c=1}^C \left(x_t - V_c^k(x_t)\right)
\end{equation}
where $\eta > 0$ is the server learning rate. Note that when $\eta = 1$, we directly recover \cref{alg:fedavg}. In the sequel we let \fedavg denote the update rule in \eqref{eq:fedavg_update}. If we let $V_s$ be the ``server'' vector field given by
\begin{equation}\label{eq:server_vector_field}
V_s = \dfrac{1}{C}\sum_{c=1}^C (I - V_c^k)
\end{equation}
then \eqref{eq:fedavg_update} is equivalent to
\begin{equation}\label{eq:fedavg_update_alt}
x_{t+1} = x_t-\eta V_s(x_t).
\end{equation}
If each $V_c$ is $k$-conservative, then $V_s$ is an average of conservative vector fields and is conservative as well. Therefore, there is some function $f_s$ such that $\nabla f_s = V_s$, and \eqref{eq:fedavg_update_alt} is equivalent to $x_{t+1} = x_t - \eta\nabla f_s(x_t)$. This is exactly gradient descent on the ``surrogate loss'' $f_s$. This leads us to our guiding observation.

\begin{center}
    \textit{If each $V_c$ is $k$-conservative, then \fedavg is equivalent to gradient descent on some surrogate loss function.}
\end{center}

A special case of this observation was first made and utilized by \citet{charles2021convergence} in the setting that each $f_c$ is a quadratic function. In this work, we consider more general functions, including some non-convex functions.

\subsection{Non-Conservative Dynamics in Federated Learning}\label{sec:non_conservative_dynamics}

As discussed above, when the vector fields $I-\gamma\nabla f_c$ are $k$-conservative, \fedavg with $k$ local steps behaves identically to gradient descent on some surrogate loss. In this section we show that in the absence of this $k$-conservatism, \fedavg can demonstrate fundamentally non-conservative behavior, making its dynamics distinct from those of gradient descent. Notably, this can occur even when $C = 2$ and in fully deterministic settings.

For example, for $c \in \{1, 2\}$, consider the client loss functions
\begin{equation}\label{eq:non_conservative_client_eqs}
f_c(x, y) := f_c^{(1)}(x, y) + f_c^{(2)}(x, y)
\end{equation}
where
\[
f_c^{(1)}(x, y) := \min\left(\frac{\alpha_c}{2}\left(y - y_c\right)^2 +\frac{\beta_c}{2}\left(x - x_c\right)^2, 1\right), 
\]
\[
f_c^{(2)}(x, y) := 
\min\left(\frac{\alpha_c}{2}\left(y + y_c\right)^2 + \frac{\beta_c}{2}\left(x + x_c\right)^2, 1\right).
\]
Here, $\alpha_c, \beta_c \in \R$, $x_c, y_c \in \R^2$ are fixed. Notably, $I - \gamma\nabla f_c$ may not be $k$-conservative for $k > 1$. As we show in \cref{appendix:cyclic_fl}, for some choice of $\alpha_c, \beta_c\in \R$, $x_c, y_c \in \R^2$ (for $c = 1, 2$), $\gamma > 0$ and $k$ sufficiently large, the resulting server vector field $V_s(x, y)$ in \eqref{eq:server_vector_field} is non-conservative.

To help illustrate this, we plot this non-conservative server vector field $V_s(x, y)$ in \cref{fig:cyclic_server_vf_main}. Note there is a region of initial points $(x_0, y_0)$ under which the dynamics of \fedavg are entirely circular and periodic, as long as $\eta$ is sufficiently small. In short, \fedavg may behave badly in the absence of $k$-conservatism.

\begin{figure}[t]
\caption{Two-dimensional non-conservative server vector field $V_s(x, y)$ induced by $f_1, f_2$ in \eqref{eq:non_conservative_client_eqs} for $k$ sufficiently large.}
\includegraphics[width=8cm]{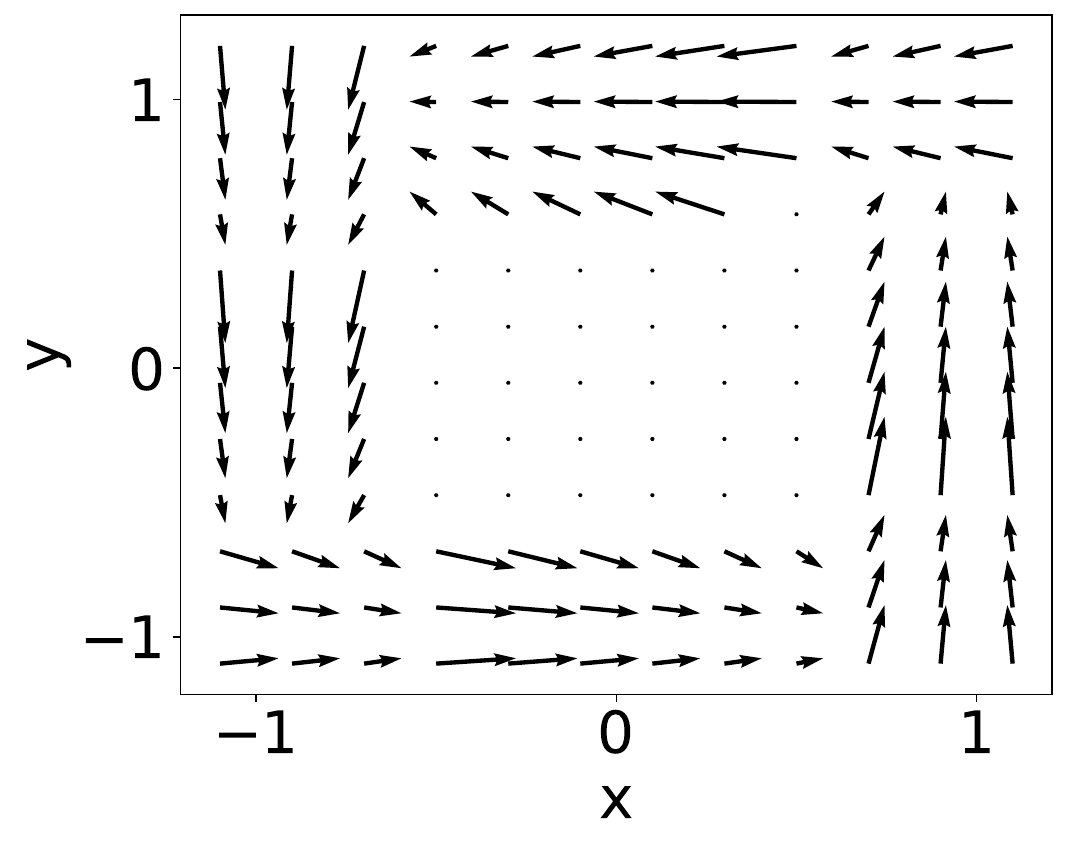}
\centering
\label{fig:cyclic_server_vf_main}
\end{figure}

\section{Examples of $k$-Conservative Vector Fields}\label{sec:examples}

We now give concrete examples of $k$-conservative vector fields. As we will show, these include vector fields associated with linear and logistic regression. Let $\mathcal{P}_d(\R^n, \R^m)$ denote the subset of $\mathcal{V}(\R^n, \R^m)$ whose coordinate functions are homogeneous polynomials of degree $d$. We abbreviate this as $\mathcal{P}_d(\R^n)$ when $n = m$. For more in-depth examples, see \cref{appendix:examples}.

\paragraph{Constant Vector Fields.} The space $\mathcal{P}_0(\R^n)$ of constant vector fields is clearly closed under self-composition. Constant vector fields are conservative, so $\mathcal{P}_0(\R^n)$ is $\infty$-conservative.

\paragraph{Affine Vector Fields.} Let $\mathcal{A}(\R^n)$ be the set of affine vector fields in $\mathcal{V}(\R^n)$. This consists of all $V$ of the form $V(x) = Ax + b$ for $A \in \R^{n \times n}$, $b \in \R^n$. Let $\mathcal{S}(\R^n)$ denote the set of such $V$ where $A$ is symmetric. Note that $\mathcal{S}(\R^n)$ is closed under self-composition. A straightforward computation shows that $V$ is conservative if and only if $A$ is symmetric. Hence, $V \in \mathcal{A}(\R^n)$ is conservative if and only if $V \in \mathcal{S}(\R^n)$, in which case it is also $\infty$-conservative. In particular, if $f$ is a quadratic function then $\nabla f$ and $I-\gamma\nabla f$ are both $\infty$-conservative.

\paragraph{Continuous Univariate Functions.} Consider the set $\mathcal{C}^0(\R)$ of continuous functions from $\R$ to $\R$. By elementary analysis, $\mathcal{C}^0(\R)$ is closed under self-composition, and by the fundamental theorem of calculus, it is conservative. Thus, $\mathcal{C}^0(\R)$ is $\infty$-conservative.

More generally, let $\mathcal{C}^0(\R)^n$ denote the subset of $\mathcal{V}(\R^n)$ containing vector fields of the form
\[
V(x_1, \dots, x_n) = (f_1(x_1), f_2(x_2), \dots, f_n(x_n))
\]
where $f_1, \dots, f_n \in \mathcal{C}^0(\R)$. Then note that
\[
V(x_1, \dots, x_n) = \nabla \left( \sum_{i=1}^n \int_0^{x_i} f_i(t)dt \right)
\]
so $\mathcal{C}^0(\R)^n$ is conservative. Since $\mathcal{C}^0(\R)^n$ is closed under self-composition, it is also $\infty$-conservative.

\paragraph{Non-example: Cubic Polynomials.} Let $f(x, y) = x^2y$. By direct computation,
\[
(\nabla f)^2(x, y) = \begin{pmatrix}4x^3y \\ 4x^2y^2 \end{pmatrix} =: \begin{pmatrix}h_1(x, y) \\ h_2(x, y)\end{pmatrix}.
\]
We then have $\frac{\partial}{\partial y}h_1(x, y) = 4x^3, \frac{\partial}{\partial x}h_2(x, y) = 8xy^2$. By Clairaut's theorem (see \cite[Chapter 4]{spivak2018calculus}), $(\nabla f)^2$ is not conservative. Thus, $\nabla \mathcal{P}_3(\R^2, \R)$ is conservative but not $2$-conservative.

\subsection{Gradient Vector Fields of Generalized Linear Models}\label{example:gradient_glms}

Let $\mathcal{G}(\R^n, \R) \subsetneq \mathcal{C}^1(\R^n, \R)$ denote the class of functions $f: \R^n \to \R$ of the form
\begin{equation}\label{eq:generalized_linear}
f(x) = \sum_{i=1}^m \sigma(\langle x, z_i\rangle)
\end{equation}
where $m$ is a positive integer, $z_i \in \R^n$, and $\sigma \in \mathcal{C}^1(\R)$. Such functions arise in statistics and optimization when learning generalized linear models. For example, when $\sigma(t) = \ln(1+ e^{-t})$, \eqref{eq:generalized_linear} is effectively the loss function used in logistic regression.

We further define $\mathcal{G}_{\perp}(\R^n, \R) \subsetneq \mathcal{G}(\R^n, \R)$ to be the set of functions of the form \eqref{eq:generalized_linear} where $\{z_i\}_{i=1}^m$ are mutually orthogonal. We then have the following result.

\begin{theorem}\label{thm:glm_li}
Let $f \in \mathcal{G}_{\perp}(\R^n, \R)$ be as in \eqref{eq:generalized_linear}. Let $\phi_i(t) = \norm{z_i}^2\sigma'(t)$. For all $k \geq 2$,
\begin{equation}\label{eq:glm_phi}
    (\nabla f)^k(x) = \nabla \left(\sum_{i=1}^m \int_0^{\langle x, z_i\rangle} \sigma'(\phi_{i}^{k-1}(t))dt \right).
\end{equation}
Thus, $\nabla\mathcal{G}_{\perp}(\R^n, \R)$ is $\infty$-conservative and closed under self-composition.
\end{theorem}

\begin{proof}
Let $V = \nabla f$. We claim that for all $k \geq 1$,
\[
    V^k(x) = \sum_{i=1}^m\sigma'(\phi_i^{k-1}(\langle x, z_i\rangle))z_i
\]
where $\phi_i^0$ is the identity function. We will show this inductively. This clearly holds for $k = 1$. We then have
\begin{align*}
    V^{k+1}(x) &= \sum_{i=1}^m \sigma'(\langle V^k(x), z_i\rangle)z_i\\
    &= \sum_{i=1}^m\sigma'\bigg(\left\langle \sum_{j=1}^m\sigma'(\phi_i^{k-1}\big(\langle x, z_j\rangle\big))z_j, z_i \right\rangle\bigg)z_i\\
    &= \sum_{i=1}^m\sigma'\bigg(\norm{z_i}^2\sigma'(\phi_i^{k-1}\big(\langle x, z_i\rangle\big))\bigg)z_i\\
    &= \sum_{i=1}^m\sigma'(\phi_i^k(\langle x, z_i\rangle))z_i.
\end{align*}

Here, the second equality follows from the inductive hypothesis, while the third follows from the orthogonality of the $z_i$. Therefore, if we define $h_k: \R \to \R$ via
\[
h_k(x) = \sum_{i=1}^m\int_0^{\langle x, z_i\rangle}\sigma'(\phi_i^{k-1}(t))dt
\]
then by the chain rule,
\[
\nabla h_k(x) = \sum_{i=1}^m \sigma'(\phi_i^{k-1}(\langle x, z_i\rangle))z_i = V^k(x).
\]
\end{proof}

In order to understand the dynamics of gradient descent on generalized linear models, we now extend Theorem \ref{thm:glm_li} to the function class $I-\gamma\nabla\mathcal{G}_{\perp}(\R^n, \R)$.

\begin{theorem}\label{thm:glm_li_gd}
Let $f \in \mathcal{G}_{\perp}(\R^n, \R)$ be as in \eqref{eq:generalized_linear} and fix $\gamma \in \R$. Let $\psi_i(t) = t - \gamma\|z_i\|^2\sigma'(t)$. For all $k \geq 2$,
\begin{equation}\label{eq:glm_psi}
    (I-\gamma\nabla f)^k(x) = x - \gamma\nabla \left(\sum_{i=1}^m \int_0^{\langle x, z_i\rangle} \sigma'(\psi_{i}^{k-1}(t))dt \right).
\end{equation}
Thus, $I-\gamma\nabla\mathcal{G}_{\perp}(\R^n, \R)$ is $\infty$-conservative and closed under self-composition.
\end{theorem}

\begin{proof}
The proof is nearly identical to the proof of Theorem \ref{thm:glm_li}. Let $V(x) = x - \gamma\nabla f(x)$. A slight modification of the inductive argument in the proof of Theorem \ref{thm:glm_li} implies that
\[
V^{k}(x) = x - \gamma\sum_{i=1}^m\sigma'(\psi_i^{k-1}(\langle x, z_i\rangle))z_i.
\]
By the chain rule, this implies that
\[
V^k(x) = x - \gamma\nabla \left(\sum_{i=1}^m \int_0^{\langle x, z_i\rangle} \sigma'(\psi_i^{k-1}(t))dt\right).
\]
\end{proof}

On the other hand, $\nabla\mathcal{G}(\R^n, \R)$ is not $2$-conservative. Let $f_1(x,y) = e^x, f_2(x, y) = e^{x+y}$, $f_3 = f_1 + f_2$. Note that by Theorem \ref{thm:glm_li}, $\nabla f_1, \nabla f_2$ are both $\infty$-conservative. However, by direct computation
\[
(\nabla f_3)^2(x, y) = \begin{pmatrix}\exp(e^x + e^{x+y}) + \exp(e^x + 2e^{x+y})\\ \exp(e^x+2e^{x+y})\end{pmatrix} =: \begin{pmatrix}h_1(x, y) \\ h_2(x, y)\end{pmatrix}.
\]
One can then verify that $\frac{\partial}{\partial y}h_1(x, y) \neq \frac{\partial}{\partial x}h_2(x,y)$,
so by Clairaut's theorem, $\nabla f_3$ is not 2-conservative. Notably, $f_1, f_2$ and $f_3$ are all convex functions, demonstrating that whether $\nabla \mathcal{F}$ is $\infty$-conservative is not determined by whether the class $\mathcal{F}$ is convex.

While $f \in \mathcal{G}_\perp(\R^n, \R)$ implies that $\nabla f$ is $\infty$-conservative, exactly characterizing the set of $\infty$-conservative vector fields in $\nabla\mathcal{G}(\R^n, \R)$ remains an open question. In particular, it is unclear whether there are any $\infty$-conservative vector fields in $\nabla \mathcal{G}(\R^n, \R)\backslash\nabla \mathcal{G}_\perp(\R^n, \R)$. Part of the difficulty in this problem comes from the fact that a function $f \in \mathcal{C}^\infty(\R^n, \R)$ can have multiple representations satisfying \eqref{eq:generalized_linear}.

\section{Smooth $k$-Conservative Vector Fields}\label{sec:smooth_vector_fields}

We now explicitly construct the space of smooth, $k$-conservative vector fields. Given $V \in \mathcal{C}^\infty(\R^n)$, let $J(V): \R^n \to \R^{n\times n}$ denote its Jacobian, which we can view as an $n\times n$ matrix over $\mathcal{C}^\infty(\R^n, \R)$.
If $V \in \mathcal{C}^\infty(\R^n)$, then by the Poincar\'e lemma \citep[Section 4.18]{diff_manifolds}, $V$ is $k$-conservative if and only if $J(V^k)$ is symmetric.
For $k \geq 1$, we then define $D_k: \mathcal{C}^\infty(\R^n) \to \mathcal{C}^\infty(\R^n, \R^{n\times n})$ by
\begin{equation}\label{eq:D_k_map}
D_k(V) := J(V^k) - J(V^k)^{\intercal}.
\end{equation}
Thus, $V \in \mathcal{C}^\infty(\R^n)$ is $k$-conservative if and only if $D_k(V) = 0$.
We may now define the space of smooth, $k$-conservative vector fields by $\mathcal{W}^k(\R^n) := D_k^{-1}(\{0\})$ and $\mathcal{W}^\infty(\R^n) := \cap_{k=1}^\infty \mathcal{W}^k(\R^n)$.
We note a few facts about $\mathcal{W}^\infty(\R^n)$:
\begin{enumerate}
\item  $\mathcal{W}^k(\R^n)$ and $\mathcal{W}^\infty(\R^n)$ are closed in $\mathcal{C}^\infty(\mathbb{R}^n)$ under several natural topologies, like that of uniform convergence of all derivatives on compact sets. To see this, note that $D_k$ is a continuous function in this topology, so $D_k^{-1}(\{0\}) = \mathcal{W}^k(\R^n)$ is closed. Thus, $\mathcal{W}^\infty(\R^n)$ is an intersection of closed sets, and is closed itself.
\item $\mathcal{W}^\infty(\R^n)$ is closed under scalar multiplication. While it contains linear subspaces (such as the space of symmetric linear vector fields, see \cref{sec:examples}), it is not closed under addition. For a simple counter-example, see the end of \cref{example:gradient_glms}.
\item While $\mathcal{W}^\infty(\R^n)$ is closed under self-composition, it is not closed under arbitrary composition. See \cref{appendix:examples} for an explicit counter-example.
\end{enumerate}

\noindent Some basic open questions on the structure of $\mathcal{W}^\infty(\R^n)$:
\begin{enumerate}
    \item How does $W^k(\R^n)$ relate to $W^j(\R^n)$ for $k \neq j$? As we show in \cref{appendix:examples}, $\mathcal{W}^k(\R^n) \not\subseteq \mathcal{W}^j(\R^n)$ for $j < k$. More generally, are there smooth vector fields that are $k$-conservative but not $j$-conservative for $j \neq k$?
    \item If we restrict to $\mathcal{P}_d(\R^n)$, the zero locus of $D_k$ defines a projective variety over the coefficients of polynomials in $\mathcal{P}_d(\R^n)$. For example, applying \eqref{eq:D_k_map} to $\mathcal{P}_d(\R^n)$, we find:
    \begin{itemize}
        \item  $\mathcal{W}^1(\R^n)\cap \mathcal{P}_1(\R^n)$ is a hyperplane.
        \item $\mathcal{W}^2(\R^2)\cap\mathcal{P}_1(\R^2)$ is a union of two hyperplanes.
        \item $\mathcal{W}^3(\R^2)\cap\mathcal{P}_1(\R^2)$ is a union of a hyperplane and a quadric surface.
        \item $\mathcal{W}^1(\R^2) \cap \mathcal{W}^2(\R^2) \cap \mathcal{P}_2(\R^2)$ is a quadric surface.
    \end{itemize}
    See \cref{appendix:examples} for the full details on these computations. Can we say anything more general? For example, what is the degree of $\mathcal{W}^k(\R^n)\cap \mathcal{P}_d(\R^n)$?
    \item For all $k \geq 1$, define $\rho_k: \mathcal{W}^\infty(\R^n) \to \mathcal{W}^\infty(\R^n)$ via $V \mapsto V^k$. Many of the discussions above can be rephrased in terms of properties of this map. For example, \cref{thm:glm_li} implies that $\rho_k$ is an endomorphism on $\nabla\mathcal{G}_\perp(\R^n, \R)$. Are there other important function classes for which $\rho_k$ is an endomorphism? More broadly speaking, we may also wish to understand the image of $\rho_k$. Note that this is important for federated learning, as according to the discussion in \cref{sec:federated_learning}, this will govern what kinds of dynamics of \fedavg are possible in settings where clients have $\infty$-conservative loss functions.
\end{enumerate}

\section{Conservatism and Lifting}\label{sec:optimization}

In this section, we show that if $V$ is $k$-conservative, then many properties of $V$ ``lift'' to the vector field $V^k$. In particular, we will show that many properties important for optimization (including convexity, smoothness, and Lipschitz-continuity) will lift under certain assumptions related to $k$-conservatism. By applying these lifting results to vector fields arising in federated learning (\cref{sec:federated_learning}), we will be able to ``lift'' convergence rates for centralized optimization algorithms to federated learning algorithms (\cref{sec:convergence}).

Note that for smooth functions, properties such as convexity can be rephrased in terms of eigenvalues of Jacobian matrices. As we will show, under $k$-conservatism, self-compositions of vector fields will yield eigenvalues that behave in predictable ways.

\begin{proposition}\label{prop:preserve_eig}
Suppose $V \in \mathcal{C}^\infty(\R^n)$ is $j$-conservative for $1 \leq j \leq k$, with $V^j = \nabla g_j$. Then for all such $j$, the function $g_j$ is smooth and satisfies:
\begin{enumerate}
    \item \label{prop:preserve_eig_second} Suppose there are $\alpha, \beta \geq 0$ such that for all $x$, $\alpha I \preceq J(V)(x) \preceq \beta I$. Then for all $x$,
    \[\alpha^kI \preceq J(\nabla g_j)(x) \preceq \beta^k I.\]
    \item \label{prop:preserve_eig_third} Suppose there is some $\lambda\geq 0$ such that for all $x$, $-\lambda I \preceq J(V)(x) \preceq \lambda I$. Then for all $x$,
\[-\lambda^kI \preceq J(\nabla g_j)(x) \preceq \lambda^k I.\]
\end{enumerate}
\cref{prop:preserve_eig_second,prop:preserve_eig_third} also hold if we change $\preceq$ to $\prec$ throughout.

\end{proposition}

\begin{proof}
Since $V^j = \nabla g_j$ (and in particular, $g_j$ is differentiable), we must have $g_j \in \mathcal{C}^\infty(\R^n, \R)$. For \cref{prop:preserve_eig_second}, we proceed inductively. For $k = 1$, the result holds by assumption. For the inductive step, let $2 \leq k \leq K$, and assume the result holds for $k-1$. Define $J_j(x) := J(\nabla g_j)(x)$, so that in particular, $J_1(x) = J(V)(x)$. By the chain rule,
\begin{equation}\label{eq:inductive_jacobian}
J_j(x) = J_1(\nabla g_{j-1}(x))J_{j-1}(x).
\end{equation}
By the inductive hypothesis, we have
\[
\alpha^{j-1} I \preceq J_{j-1}(x) \preceq \beta^{j-1} I
\]
and by our assumptions on $V$, we have
\[
\alpha I \preceq J_1(\nabla g_{j-1})(x) \preceq \beta I.
\]
Since $J_j(x)$ is symmetric (as it is the Jacobian of a gradient field), its eigenvalues are therefore products of eigenvalues of $J_{j-1}(x)$ and $J_1(\nabla g_{j-1})(x)$. Hence, its maximum eigenvalue is at most $\beta^j$, and its minimum eigenvalue is at most $\alpha^j$.

The proof of \cref{prop:preserve_eig_third} follows in a similar way, noting that by the inductive hypothesis, the matrices on the right-hand side of \eqref{eq:inductive_jacobian} will have eigenvalues in the ranges of $[-\lambda, \lambda]$ and $[-\lambda^{j-1}, \lambda^{j-1}]$. Since $J_j(x)$ is symmetric, its eigenvalues are products of the eigenvalues of the matrices in the right-hand side of \eqref{eq:inductive_jacobian}, and the result follows.
\end{proof}

\begin{remark}
Note the critical role of symmetry in the argument above. In $\R^n$, $J(V^k)$ is symmetric if and only if $V$ is $k$-conservative. Thus, $k$-conservatism is exactly the condition required for us to reason about how the eigenvalues of $J(V^k)$ relate to that of $J(V)$.
\end{remark}

We will use Proposition \ref{prop:preserve_eig} to show that iterating $\infty$-conservative vector fields preserves geometric properties, including Lipschitz continuity, as in the following definition.

\begin{definition}
A vector field $V \in \mathcal{C}^1(\R^n)$ is $\beta$-Lipschitz continuous if for all $x \in \R^n$, $\|J(V)(x)\| \leq \beta$. $V$ is Lipschitz continuous if there is some $\beta$ for which $V$ is $\beta$-Lipschitz continuous.
\end{definition}

In the definition above, $\|\cdot\|$ is the operator norm induced by the $\ell_2$ norm on $\R^n$, viewing $J(V)(x)$ as an $n\times n$ matrix over $\R$. In the following, we let $\mathcal{L}(\R^n) \subsetneq \mathcal{V}(\R^n)$ denote the set of Lipschitz continuous vector fields. Proposition \ref{prop:preserve_eig} implies the following result.

\begin{corollary}\label{cor:preserve_convex}
Let $\mathcal{F} \subsetneq \mathcal{C}^\infty(\R^n, \R)$ be the set of (a) smooth, strongly convex functions, (b) smooth, strictly convex functions, or (c) smooth, convex functions. Then $\nabla\mathcal{F}\cap\mathcal{W}^\infty(\R^n)$ and $\nabla\mathcal{F}\cap\mathcal{W}^\infty(\R^n)\cap\mathcal{L}(\R^n)$ are closed under self-composition.
\end{corollary}

\begin{proof}
This follows directly from Proposition \ref{prop:preserve_eig} by setting $V = \nabla f$ for $f \in \mathcal{F}$. For (a), if $f$ is smooth and strongly convex, then there is some $\alpha > 0$ such that $\alpha I \preceq J(\nabla f)(x)$ for all $x$. Since $\nabla f \in \mathcal{W}^\infty(\R^n)$, for all $k \geq 1$, there is some $g_k$ such that $\nabla g_k = (\nabla f)^k$. By Proposition \ref{prop:preserve_eig}, we have $\alpha^k I \preceq J(\nabla g_k)(x)$, so $g_k$ is smooth and strongly convex. If $\nabla f$ is also Lipschitz continuous, then there is some $\beta$ for which $J(\nabla f)(x) \preceq \beta I$ for all $x$, and a similar argument shows that $\alpha^k I \preceq J(\nabla g_k)(x) \preceq \beta^k I$.

The convex and strictly convex cases follow in an analogous manner, as they correspond respectively to the bounds $0 \preceq J(\nabla f)(x)$ and $0 \prec J(\nabla f)(x)$, which are preserved under $k$-fold composition by Proposition \ref{prop:preserve_eig}.
\end{proof}

Thus, convexity "lifts" under self-composition of the associated gradient vector field: If $f$ is smooth, convex, and $\nabla f$ is $j$-conservative for $1 \leq j \leq k$, then $(\nabla f)^k = \nabla g$ for some smooth, convex function $g$.

Next, we consider vector fields of the form $V = I-(I-\gamma \nabla f)^k$ where $\gamma > 0$. Note that such vector fields arise naturally in the context of federated learning, as in \eqref{eq:server_vector_field}. In the following lemma, we show that if $V$ is $\infty$-conservative and $V^k = \nabla h_k$, then $h_k$ inherits smoothness and critical points from $f$.

\begin{lemma}\label{lemma:client_properties}
Let $f \in \mathcal{C}^\infty(\R^n, \R)$ and $\gamma \in \R_{> 0}$. Suppose that $I-\gamma\nabla f$ is $j$-conservative for $1 \leq j \leq k$. Then $V_k := I - (I-\gamma\nabla f)^k$ is conservative. Furthermore, if $\nabla h_k = V_k$ then:
\begin{enumerate}
    \item $h_k$ is smooth.
    \item If $\nabla f(y) = 0$, then $\nabla h_k(y) = 0$.
\end{enumerate}
\end{lemma}

\begin{proof}
For (1), $h_k$ is differentiable by assumption. Moreover, $\nabla h_k = V_k \in \mathcal{C}^\infty(\R^n)$, as smoothness is preserved under addition and composition. Hence, $h_k \in \mathcal{C}^\infty(\R^n, \R)$. For (2), note that since $\nabla f(y) = 0$, we have
\[(I-\gamma\nabla f)(y) = y -\gamma\nabla f(y) = y\]
This implies that $(I-\gamma \nabla f)^k(y) = y$, so that $\nabla h_k(y) = y - (I - \gamma\nabla f)^k(y) = 0$.\end{proof}

In fact, many geometric properties important to optimization (such as convexity) are also inherited by $h_k$, provided that $\gamma$ is not too large, as in the following.

\begin{lemma}\label{lemma:preserve_properties}
Suppose $f \in \mathcal{C}^\infty(\R^n, \R)$ and $\nabla f$ is $\beta$-Lipschitz continuous. Suppose that for some $\gamma \in \R_{> 0}$, $I-\gamma\nabla f$ is $j$-conservative for $1 \leq j \leq k$, with $\nabla h_k = I - (I-\gamma\nabla f)^k$. Then:
\begin{enumerate}
    \item If $f$ is $\alpha$-strongly convex and $\gamma \leq 2(\alpha+\beta)^{-1}$ then $h_k$ is $(1-\lambda^k)$-strongly convex and $\nabla h_k$ is $(1+\lambda^k)$-Lipschitz continuous for $\lambda = 1-\gamma\alpha$.
    \item If $f$ is convex and $\gamma \leq 2\beta^{-1}$ then $h_k$ is convex and $\nabla h_k$ is $2$-Lipschitz continuous. If $\gamma \leq \beta^{-1}$, then $\nabla h_k$ is $1$-Lipschitz continuous.
    \item If $f$ is strictly convex and $\gamma < 2\beta^{-1}$ then $h_k$ is strictly convex.
    \item If $f$ is $\delta$-weakly convex for $\delta \leq \beta$ and $\gamma \leq 2\beta^{-1}$, then $h_k$ is $(\lambda^k - 1)$-weakly convex and $\nabla h_k$ is $(1+\lambda^k)$-Lipschitz continuous for $\lambda = 1+\gamma\delta$.
\end{enumerate}
\end{lemma}

\begin{proof}
This is a direct consequence of Proposition \ref{prop:preserve_eig}. For (1), by assumption we have $\alpha I \preceq J(\nabla f)(x) \preceq \beta I$ for all $x$, and therefore $-\lambda \preceq J(I-\gamma\nabla f)(x) \preceq \lambda I$ for all $x$ where $\lambda = 1-\gamma \alpha$. By Proposition \ref{prop:preserve_eig}, we have that for all $x$
\[
-\lambda^kI \preceq J((I-\gamma\nabla f)^k)(x) \preceq \lambda^kI
\]
and so
\[
0 \prec (1-\lambda^k)I \preceq J(\nabla h_k)(x) \preceq (1 + \lambda^k) I.
\]

The remaining parts of the lemma are proved in an analogous way using Proposition \ref{prop:preserve_eig} and basic algebraic manipulations.
\end{proof}

\section{Convergence Rates in Federated Learning}\label{sec:convergence}

We now use our machinery above to analyze the convergence of \fedavg in various settings. Recall that the server update at each round is given by $x_{t+1} = x_t - \eta V_s(x_t)$, where the ``server vector field'' $V_s$ is given by \eqref{eq:server_vector_field}. Throughout, we assume that each client $c$ performs $k$ steps of gradient descent with learning rate $\gamma > 0$ on their loss function $f_c$. As sketched in \cref{sec:federated_learning}, when the client losses are all $k$-conservative, we have the following link between \fedavg and gradient descent.

\begin{theorem}\label{thm:server_conservative}
Suppose that for all $c$, $f_c$ is a differentiable function such that $I-\gamma\nabla f_c$ is $k$-conservative. Then $V_s$ is a conservative vector field. In particular, there is some function $f_s$ such that $V_s = \nabla f_s$ and the \fedavg server update in \eqref{eq:fedavg_update} is equivalent to
\begin{equation}\label{eq:server_grad_descent}
    x_{t+1} = x_t - \eta\nabla f_s(x_t).
\end{equation}
\end{theorem}

\begin{proof}
By assumption, for $c = 1, \dots, C$, there is some function $h_c$ such that $\nabla h_c = (I-\gamma\nabla f_c)^k$. We can then define $q_c: \R^n \to \R$ by $q_c(x) := \frac{1}{2}\norm{x}^2 - h_c(x)$.
By construction,
\[\nabla q_c = I - \nabla h_c = I - (I-\gamma\nabla f)^k\]
implying that $V_s = C^{-1}\sum_{c=1}^C \nabla q_c$. Therefore, $V_s = \nabla f_s$ where $f_s = C^{-1}\sum_{c=1}^C q_c$.
\end{proof}

Note that in general, $f_s$ need not equal the average $f_{avg}$ of the client loss functions. If we have some understanding of $f_s$ (for example, whether $f_s$ is convex), we can immediately apply centralized optimization results to derive convergence results for \fedavg. To better understand the structure of $f_s$, we will use Lemma \ref{lemma:preserve_properties}. However, this requires $j$-conservatism for $j = 1, \dots, k$, as well as Lipschitz continuity. Thus, we make the following assumptions.

\begin{assumption}\label{assm:client_conservative}
For all $c$, $f_c$ is smooth and $I-\gamma\nabla f_c$ is $j$-conservative for $1 \leq j \leq k$.
\end{assumption}

\begin{assumption}\label{assm:beta_lipschitz}
For all $c$, $\nabla f_c$ is $\beta$-Lipschitz continuous.
\end{assumption}

Under Assumptions \ref{assm:client_conservative} and \ref{assm:beta_lipschitz}, Lemma \ref{lemma:preserve_properties} lifts geometric properties of the client loss functions $f_c$ to the function $f_s$.
Combining this with Theorem \ref{thm:server_conservative}, we can translate convergence rates for gradient descent to convergence rates for \fedavg in strongly convex and convex settings. We make no direct assumptions on client heterogeneity. Throughout, we let $f_s$ be a function such that $V_s = \nabla f_s$, as guaranteed by Theorem \ref{thm:server_conservative}.

\begin{theorem}\label{thm:fedavg_convergence_sc}
Suppose Assumptions \ref{assm:client_conservative} and \ref{assm:beta_lipschitz} hold, and that for all $c$, $f_c$ is $\alpha$-strongly convex. Then $f_s$ has a unique minimizer $x_s^*$, and if $\gamma = 2(\alpha+\beta)^{-1}$, $\eta = 1$, then the iterates $\{x_t\}_{t=0}^\infty$ of \fedavg satisfy
\begin{equation}\label{eq:fedavg_convergence_eq}
\norm{x_{t} - x_s^*} \leq \left(\dfrac{\beta - \alpha}{\beta + \alpha}\right)^{kt}\norm{x_0 - x_s^*}.
\end{equation}
\end{theorem}

\begin{proof}
This follows directly by combining Theorem \ref{thm:server_conservative} and Lemma \ref{lemma:preserve_properties} with well-known convergence rates for smooth, strongly convex functions (for example, see \cite[Theorem 2.1.15]{nesterov2003introductory}). See \cref{appendix:proof_sc} for more details.
\end{proof}

The convergence rate in \eqref{eq:fedavg_convergence_eq} was shown first by \citet[Theorem 2.11]{malinovskiy2020local}, whose result also applies to non-conservative gradient vector fields. The salient difference is that under under our assumptions, the limit point $x_s^*$ is actually the global minimizer of some strongly convex function. As we discuss below, this allows us to immediately derive analogous results for variants of \fedavg that apply other server optimizers.

When $k = 1$, \cref{thm:fedavg_convergence_sc} recovers the convergence rate of gradient descent on $f_{avg}$. Hence, \fedavg with $k > 1$ yields an exponential improvement in convergence (with respect to $k$), but may not converge to the minimizer $x^*$ of $f_{avg}$. To understand this discrepancy, one could analyze $\|x_s^* - x^*\|$. A tight upper bound was given for strongly convex quadratic functions by \citet[Lemma 5]{charles2021convergence}. A bound in the general strongly convex setting (not assuming $k$-conservatism) was given by \citet[Theorem 2.14]{malinovskiy2020local}, though whether this bound can be improved under \cref{assm:client_conservative} is an open question.

We now give a convergence rate for \fedavg in the convex setting.
\begin{theorem}\label{thm:fedavg_convergence_convex}
Suppose Assumptions \ref{assm:client_conservative} and \ref{assm:beta_lipschitz} hold, and that for all $c$, $f_c$ is convex with finite minimizer. Then $f_s$ has a finite minimizer $x_s^*$, and if $\gamma = \beta^{-1}$, $\eta = 1$, then the iterates $\{x_t\}_{t=0}^\infty$ of \fedavg satisfy
\begin{equation}\label{eq:fedavg_convergence_convex}
f_s(x_t) - f_s(x_s^*) \leq \dfrac{1}{2t}\norm{x_0 - x_s^*}^2.
\end{equation}
\end{theorem}

\begin{proof}
This follows by combining Theorem \ref{thm:server_conservative} and Lemma \ref{lemma:preserve_properties} with well-known convergence rates for smooth, convex functions (for example, see \cite[Theorem 3.3]{bubeck2015convex}). See \cref{appendix:proof_convex} for more details.
\end{proof}

To the best of our knowledge, Theorem \ref{thm:fedavg_convergence_convex} is the first result showing that \fedavg exhibits convergent behavior on a class of (non-strongly) convex functions, even with fixed learning rates and $k > 1$. Unlike Theorem \ref{thm:fedavg_convergence_sc}, it is not clear that the convergence in \eqref{eq:fedavg_convergence_convex} is ``faster'' (in some sense) than the convergence of gradient descent on $f_{avg}$. Such analysis is an open and important problem.

\subsection{Extensions to Other Methods}

Above, we showed that our results from \cref{sec:optimization} allow us to transfer classical convergence rates for gradient descent to convergence rates for \fedavg (under $k$-conservatism). However, much of our machinery (in particular, our lifting results, such as \cref{lemma:preserve_properties}) is not specific to the server update \eqref{eq:fedavg_update_alt} of \fedavg. In fact, our machinery will allow us to analyze any federated learning algorithm where the server update in \eqref{eq:fedavg_update_alt} is replaced with some other first-order optimization method (as proposed by \citet{reddi2021adaptive}).

In more detail, let us treat $V_s(x_t)$ as an estimate of the gradient of the loss function $f_{avg}$. If we apply gradient descent, we arrive at the update step in \eqref{eq:fedavg_update_alt}. However, we could use any first-order ``server optimization'' method \serveropt. This allows us to generalize the server update \eqref{eq:fedavg_update_alt} via the following discrete-time dynamical system:
\begin{equation}\label{eq:fedopt_update}
    x_{t+1} = \serveropt(V_s(x_t)).
\end{equation}
For example, \serveropt could be gradient descent with momentum or an adaptive method such as Adagrad~\citep{duchi2011adaptive,mcmahan2010adaptive}. These two choices of \serveropt lead to \fedavgm~\citep{hsu2019measuring} and \fedadagrad~\citep{reddi2021adaptive} respectively, and can lead to improved empirical convergence.

Under \cref{assm:client_conservative}, \eqref{eq:fedopt_update} becomes $x_{t+1} = \serveropt(\nabla f_s(x_t))$, which is equivalent to applying the first-order optimizer \serveropt to the surrogate loss $f_s$. Thus, convergence rates for \serveropt can be translated into converge rates for \eqref{eq:fedopt_update}. Notably, this implies that in some settings, there are algorithms which converge to the same point as \fedavg, but faster.

For example, in the same settings as \cref{thm:fedavg_convergence_sc}, we can improve convergence by using gradient descent with heavy-ball momentum. By an almost identical proof to \cref{thm:fedavg_convergence_sc}, we have the following result.
\begin{theorem}\label{thm:heavy_ball_convergence_sc}
Let $\{x_t\}_{t=0}^\infty$ be the iterates of \eqref{eq:fedopt_update} where \serveropt is gradient descent with heavy-ball momentum. Under the same setting as \cref{thm:fedavg_convergence_sc}, for some choice of parameters of \serveropt, the sequence $\{x_t\}_{t=0}^\infty$ satisfies
\begin{equation}\label{eq:heavy_ball_convergence_eq}
\norm{x_{t} - x_s^*} \leq \left(\dfrac{\sqrt{\kappa} - 1}{\sqrt{\kappa} + 1}\right)^t\norm{x_0 - x_s^*}
\end{equation}
where $\kappa = (1+\lambda^k)/(1-\lambda^k)$ and $\lambda = (\beta-\alpha)/(\beta+\alpha)$.
\end{theorem}

\begin{proof}
The proof is the same as for \cref{thm:fedavg_convergence_sc}, but we apply convergence rates for gradient descent with heavy-ball momentum instead (see \citep{polyak1964some}). See \cref{appendix:proof_sc} for more details.
\end{proof}

One can verify that the convergence rate in \eqref{eq:heavy_ball_convergence_eq} is faster than \eqref{eq:fedavg_convergence_eq}. We stress that while the same kind of result can be derived for any number of centralized optimization algorithms, the key point is that our analytic framework allows us to leverage existing knowledge of centralized optimization methods in the context of federated learning. In particular, this can enable more informed, theoretically grounded decisions about which choice of optimizer and hyperparameters to use in \eqref{eq:fedopt_update}.

\section{Summary and Open Problems}

Our goal above was to plainly introduce the notion of $k$-conservative vector fields and illustrate their importance to optimization and federated learning. Notably, when the clients' gradient vector fields are $k$-conservative, \fedavg is equivalent to gradient descent on some surrogate loss function (\cref{sec:federated_learning} and \cref{thm:server_conservative}). By contrast, in the absence of $k$-conservatism, \fedavg can exhibit non-convergent, circular behavior (\cref{sec:non_conservative_dynamics}). We gave some notable examples of $k$-conservative vector fields (\cref{sec:examples}) and constructed the space of smooth $k$-conservative vector fields (\cref{sec:smooth_vector_fields}). This viewpoint allowed us to show that important function properties (including convexity) lift from the client loss functions to the surrogate loss (\cref{sec:optimization}). This in turn let us leverage existing optimization theory to easily understand the convergence of federated optimization methods (\cref{sec:convergence}).

We believe that this work asks more questions than it solves, both within the realm of federated learning and without. We provide a non-comprehensive list of relevant open problems below. These vary from more abstract (for example, understanding the structure of $W^k(\R^n)\cap\mathcal{P}_d(\R^n)$ as a projective variety) to more concrete (for example, using these insights to design improved federated learning algorithms). They also span topics in geometry, dynamical systems, and optimization. While we attempt to group these open problems according to the viewpoint in which they are most natural, these viewpoints are mutually reinforcing rather than mutually exclusive, and most of these questions can be viewed from more than one perspective.

\subsection{The Geometric Perspective}

As we discuss in \cref{sec:smooth_vector_fields}, much of this work can be phrased in terms of natural questions about the geometric structure of $\mathcal{W}^\infty(\R^n)$ . Its subspaces $\mathcal{W}^k(\R^n)$ yield non-trivial algebraic-geometric objects when restricted to homogeneous polynomials, but we have only scratched the surface of understanding these spaces. Such analysis may yield practical results; deriving membership criteria for $\mathcal{W}^\infty(\R^n)$ may allow federated learning practitioners to better select and design loss functions for optimization.

The discussion above is fundamentally tied to the Euclidean setting. However, many of the questions we pose may also be applied to more general geometric objects, especially smooth Riemannian manifolds. Rather than analyzing the conservatism of vector fields, we could instead analyze the exactness of differential 1-forms. However, even defining the correct analog of being $k$-conservative in this setting is non-trivial, as we cannot arbitrarily compose sections of the cotangent bundle of a manifold. 

Finally, we focused primarily on infinitely smooth functions defined on the entirety of $\R^n$. We can, of course, define non-smooth vector fields, or vector fields whose domain is a subset of $\R^n$. Indeed, this is motivated by practice, as many functions of interest to optimization and machine learning are non-smooth or not defined globally. In such cases, analyzing whether such vector fields are in fact the gradient field of some loss function becomes more challenging, as the Poincar\'e lemma need not apply.

\subsection{The Optimization Perspective}

While $k$-conservatism of client loss functions implies that \fedavg converges in many settings, it is not strictly necessary~\citep{malinovskiy2020local}. Better characterizations of when \fedavg exhibits convergent behavior (or fails to do so) is an important open problem. Similarly, we have only scratched the surface on how the dynamics of the client loss functions lift to the server dynamics. Although many convexity-adjacent properties lift (Lemma \ref{lemma:preserve_properties}), other natural properties (including being bounded below) do not lift. What about properties such as the Polyak-\L{}ojasiewicz condition~\citep{karimi2020}? What can we say about the server loss $f_s$ in relation to the client loss functions $f_c$?

Another related open problem is understanding the empirical effectiveness of methods such as \fedavg, both in terms of convergence rates and utility of the point converged to. As discussed in \citet{wang2021field}, theoretical convergence rates of federated learning methods often do not improve upon centralized rates for algorithms such as stochastic gradient descent. While Theorem \ref{thm:fedavg_convergence_sc} shows that \fedavg accelerates convergence to a non-optimal point, it is unclear whether Theorem \ref{thm:fedavg_convergence_convex} implies a similar acceleration. Notably, very little can currently be said about the properties of this non-optimal point outside of limited settings. Is there some sense in which the limit point $x_s^*$ is a useful point of convergence, either for learning a global model, or as a starting point for personalization? More generally, are there underlying trade-offs between the accuracy and the convergence of federated optimization methods? If so, how do we effectively balance the two in practical settings?

\subsection{The Dynamical Systems Perspective}

In \fedavg, the induced server vector field in \eqref{eq:fedavg_update} need not be conservative. Regardless, it defines a discrete-time dynamical system, a system whose behavior is not entirely determined by its representability as gradient descent on some surrogate loss function. More general methods of characterizing the dynamics of this system, such as determining whether it converges, and if so to what point, would greatly benefit the analysis and design of federated learning algorithms.

This dynamical system has a number of similarities to dynamical systems defined by multi-agent interactions, as the client updates may conflict with one another. Such systems (for example, dynamical systems arising from multi-player differentiable games, such as when training generative adversarial networks~\citep{goodfellow2014generative}) may have non-zero curl, or even support compact integral curves (ruling out the existence of a Lyapunov function). Can we use insights from training multi-agent systems to create better federated learning methods? Can we classify what kinds of multi-player games arise from federated learning algorithms?

\bibliography{references}

\begin{thebibliography}{19}
\providecommand{\natexlab}[1]{#1}
\providecommand{\url}[1]{\texttt{#1}}
\expandafter\ifx\csname urlstyle\endcsname\relax
  \providecommand{\doi}[1]{doi: #1}\else
  \providecommand{\doi}{doi: \begingroup \urlstyle{rm}\Url}\fi

\bibitem[Bonawitz et~al.(2019)Bonawitz, Eichner, Grieskamp, Huba, Ingerman,
  Ivanov, Kiddon, Konečný, Mazzocchi, McMahan, Overveldt, Petrou, Ramage, and
  Roselander]{bonawitz2019towards}
K.~A. Bonawitz, Hubert Eichner, Wolfgang Grieskamp, Dzmitry Huba, Alex
  Ingerman, Vladimir Ivanov, Chloé~M Kiddon, Jakub Konečný, Stefano
  Mazzocchi, Brendan McMahan, Timon~Van Overveldt, David Petrou, Daniel Ramage,
  and Jason Roselander.
\newblock Towards federated learning at scale: System design.
\newblock In \emph{SysML 2019}, 2019.
\newblock URL \url{https://arxiv.org/abs/1902.01046}.

\bibitem[Bubeck(2015)]{bubeck2015convex}
S{\'e}bastien Bubeck.
\newblock Convex optimization: Algorithms and complexity.
\newblock \emph{Foundations and Trends{\textregistered} in Machine Learning},
  8\penalty0 (3-4):\penalty0 231--357, 2015.

\bibitem[Charles and Kone\v{c}n\'y(2021)]{charles2021convergence}
Zachary Charles and Jakub Kone\v{c}n\'y.
\newblock Convergence and accuracy trade-offs in federated learning and
  meta-learning.
\newblock In \emph{Proceedings of The 24th International Conference on
  Artificial Intelligence and Statistics}, 2021.

\bibitem[Duchi et~al.(2011)Duchi, Hazan, and Singer]{duchi2011adaptive}
John Duchi, Elad Hazan, and Yoram Singer.
\newblock Adaptive subgradient methods for online learning and stochastic
  optimization.
\newblock \emph{Journal of machine learning research}, 12\penalty0 (7), 2011.

\bibitem[Goodfellow et~al.(2014)Goodfellow, Pouget-Abadie, Mirza, Xu,
  Warde-Farley, Ozair, Courville, and Bengio]{goodfellow2014generative}
Ian Goodfellow, Jean Pouget-Abadie, Mehdi Mirza, Bing Xu, David Warde-Farley,
  Sherjil Ozair, Aaron Courville, and Yoshua Bengio.
\newblock Generative adversarial nets.
\newblock \emph{Advances in neural information processing systems}, 27, 2014.

\bibitem[Hsu et~al.(2019)Hsu, Qi, and Brown]{hsu2019measuring}
Tzu-Ming~Harry Hsu, Hang Qi, and Matthew Brown.
\newblock Measuring the effects of non-identical data distribution for
  federated visual classification.
\newblock \emph{arXiv preprint arXiv:1909.06335}, 2019.

\bibitem[Karimi et~al.(2016)Karimi, Nutini, and Schmidt]{karimi2020}
Hamed Karimi, Julie Nutini, and Mark Schmidt.
\newblock Linear convergence of gradient and proximal-gradient methods under
  the {P}olyak-{{\L}}ojasiewicz condition.
\newblock In Paolo Frasconi, Niels Landwehr, Giuseppe Manco, and Jilles
  Vreeken, editors, \emph{Machine Learning and Knowledge Discovery in
  Databases}, pages 795--811, Cham, 2016. Springer International Publishing.
\newblock ISBN 978-3-319-46128-1.

\bibitem[Khaled et~al.(2019)Khaled, Mishchenko, and
  Richt{\'a}rik]{khaled2019first}
Ahmed Khaled, Konstantin Mishchenko, and Peter Richt{\'a}rik.
\newblock First analysis of local gd on heterogeneous data.
\newblock \emph{arXiv preprint arXiv:1909.04715}, 2019.

\bibitem[Malekmohammadi et~al.(2021)Malekmohammadi, Shaloudegi, Hu, and
  Yu]{malekmohammadi2021operator}
Saber Malekmohammadi, Kiarash Shaloudegi, Zeou Hu, and Yaoliang Yu.
\newblock An operator splitting view of federated learning.
\newblock \emph{arXiv preprint arXiv:2108.05974}, 2021.

\bibitem[Malinovskiy et~al.(2020)Malinovskiy, Kovalev, Gasanov, Condat, and
  Richtarik]{malinovskiy2020local}
Grigory Malinovskiy, Dmitry Kovalev, Elnur Gasanov, Laurent Condat, and Peter
  Richtarik.
\newblock From local {SGD} to local fixed-point methods for federated learning.
\newblock In \emph{International Conference on Machine Learning}, pages
  6692--6701. PMLR, 2020.

\bibitem[McMahan et~al.(2017)McMahan, Moore, Ramage, Hampson, and
  y~Arcas]{mcmahan2017communication}
Brendan McMahan, Eider Moore, Daniel Ramage, Seth Hampson, and Blaise~Aguera
  y~Arcas.
\newblock Communication-efficient learning of deep networks from decentralized
  data.
\newblock In \emph{Artificial intelligence and statistics}, pages 1273--1282.
  PMLR, 2017.

\bibitem[McMahan and Streeter(2010)]{mcmahan2010adaptive}
H~Brendan McMahan and Matthew Streeter.
\newblock Adaptive bound optimization for online convex optimization.
\newblock \emph{arXiv preprint arXiv:1002.4908}, 2010.

\bibitem[Nesterov(2003)]{nesterov2003introductory}
Yurii Nesterov.
\newblock \emph{Introductory lectures on convex optimization: A basic course},
  volume~87.
\newblock Springer Science \& Business Media, 2003.

\bibitem[Pathak and Wainwright(2020)]{pathak2020fedsplit}
Reese Pathak and Martin~J Wainwright.
\newblock Fedsplit: an algorithmic framework for fast federated optimization.
\newblock \emph{Advances in Neural Information Processing Systems},
  33:\penalty0 7057--7066, 2020.

\bibitem[Polyak(1964)]{polyak1964some}
Boris~T Polyak.
\newblock Some methods of speeding up the convergence of iteration methods.
\newblock \emph{Ussr computational mathematics and mathematical physics},
  4\penalty0 (5):\penalty0 1--17, 1964.

\bibitem[Reddi et~al.(2021)Reddi, Charles, Zaheer, Garrett, Rush,
  Kone{\v{c}}n{\'y}, Kumar, and McMahan]{reddi2021adaptive}
Sashank~J. Reddi, Zachary Charles, Manzil Zaheer, Zachary Garrett, Keith Rush,
  Jakub Kone{\v{c}}n{\'y}, Sanjiv Kumar, and Hugh~Brendan McMahan.
\newblock Adaptive federated optimization.
\newblock In \emph{International Conference on Learning Representations}, 2021.
\newblock URL \url{https://openreview.net/forum?id=LkFG3lB13U5}.

\bibitem[Spivak(2018)]{spivak2018calculus}
Michael Spivak.
\newblock \emph{Calculus on manifolds: a modern approach to classical theorems
  of advanced calculus}.
\newblock CRC press, 2018.

\bibitem[Wang et~al.(2021)Wang, Charles, Xu, Joshi, McMahan, Al-Shedivat,
  Andrew, Avestimehr, Daly, Data, et~al.]{wang2021field}
Jianyu Wang, Zachary Charles, Zheng Xu, Gauri Joshi, H~Brendan McMahan, Maruan
  Al-Shedivat, Galen Andrew, Salman Avestimehr, Katharine Daly, Deepesh Data,
  et~al.
\newblock A field guide to federated optimization.
\newblock \emph{arXiv preprint arXiv:2107.06917}, 2021.

\bibitem[Warner(1983)]{diff_manifolds}
Frank W. (Frank~Wilson) Warner.
\newblock Foundations of differentiable manifolds and {L}ie groups, 1983.

\end{thebibliography}

\appendix

\section{In-Depth Examples}\label{appendix:examples}

In this section, we give some in-depth examples regarding the $k$-conservatism of vector fields in $\mathcal{C}^\infty(\R^2)$. Note that for $V \in \mathcal{C}^\infty(\R^2)$, $D_k(V)$, as defined in \eqref{eq:D_k_map}, is a $2\times 2$ anti-symmetric matrix over $\mathcal{C}^\infty(\R^2, \R)$. Thus, when setting $D_k(V) = 0$, it suffices to consider a single off-diagonal entry. In a slight abuse of notation, in this section we will identify $D_k(V)$ with either off-diagonal entry of $D_k(V)$. Note that this is well-defined up to a factor of $-1$.

\subsection{Linear Vector Fields}\label{example:linear_revisited}

Recall that $\mathcal{P}_1(\R^n)$ denotes the set of linear vector fields. Let $V \in \mathcal{P}_1(\R^n)$ be of the form $V(x, y) = (ax + by, cx + dy)$. Then we have the following equations (where we consider only the non-zero off-diagonal entries of $D_k$):
\begin{align*}
    D_1(V) &= b - c \\
    D_2(V) &= (b - c)(a + d) \\
    D_3(V) &= (b - c)(a^{2} + a d + b c + d^{2}) \\
    D_4(V) &= (b - c)(a + d)(a^{2} + 2 b c + d^{2}).
\end{align*}
If $V$ is conservative, then $b = c$ and these equations all vanish. Comparing $D_1$, $D_2$, and $D_3$, we see that $2$-conservative vector fields need not be conservative nor $3$-conservative. For example, if we take
\[
V(x) = \begin{pmatrix}1 & 2\\ 1 & -1\end{pmatrix}x
\]
then $V$ is 2-conservative and $4$-conservative, but not conservative or $3$-conservative.

Note that if $\mathcal{L}(\R^n)$ is the set of symmetric linear vector fields (and hence, the set of $\infty$-conservative linear vector fields), then $\mathcal{L}(\R^n)$ is closed under self-composition, but not closed under arbitrary composition. To see this, consider the symmetric linear vector fields
\[
V_1(x) = \begin{pmatrix}0 & 1\\ 1 & 0\end{pmatrix}x,~~~V_2(x) = \begin{pmatrix}1 & 0\\0 & -1\end{pmatrix}x.
\]
Then $V_1, V_2 \in \mathcal{L}(\R^n)$. However, $V_1\circ V_2 \not\in \mathcal{L}(\R^n)$ since
\[
V_1(V_2(x)) = \begin{pmatrix}0 & -1\\ 1 & 0\end{pmatrix}x
\]
which is a non-symmetric linear map. In particular, this implies that $\mathcal{W}^\infty(\R^n)$ is not closed under arbitrary composition.

Notably, $\mathcal{P}_1(\R^n)$ contains vector fields that are $j$-conservative but not $k$-conservative for $k < j$. For $j \geq 2$, consider the vector field given by $V_j(x) = A_jx$ where
\[
A_j(x) = \begin{pmatrix}\cos(\theta_j) & \sin(\theta_j) \\ -\sin(\theta_j) & \cos(\theta_j)\end{pmatrix},~~~ \theta_j = \frac{\pi}{j}.
\]
This is the vector field that rotates vectors by an angle of $\pi/j$. Since $V_j^k$ is conservative precisely when $V_j^k$ is symmetric, $V_j^k$ is conservative if and only if $\sin(k\theta_j) = 0$. Thus, $V_j$ is $k$-conservative if and only if $j$ divides $k$.

\subsection{Gradient Vector Fields of Cubic Polynomials}\label{example:cubic_revisited}

Consider the vector space $\mathcal{P}_3(\R^2, \R)$ containing polynomials of the form
\[
f(x, y) = ax^3 + bx^2y + cxy^2 + dy^3
\]
for $a,b,c,d \in \R$. All such $f$ satisfy $D_1(\nabla f) = 0$ (as $\nabla f$ is conservative). By direct computation, taking only the off-diagonal entries of $D_k$, we get
\[
D_2(\nabla f) = g_1x^3 + g_2x^2y + g_3xy^2 + g_4y^3.
\]
for $g_1, g_2, g_3, g_4 \in \R[a, b, c, d]$ defined by
\begin{align*}
g_1 &= -4b(3 a c - b^{2} + 3 b d - c^{2}) \\
g_2 &= 4(3 a - 2 c)(3 a c - b^{2} + 3 b d - c^{2}) \\
g_3 &= 4(2 b - 3 d)(3 a c - b^{2} + 3 b d - c^{2}) \\
g_4 &= 4c(3 a c - b^{2} + 3 b d - c^{2}).
\end{align*}
One can then verify that these equations vanish simultaneously if and only if
\[
g(a,b,c,d) = 3 a c - b^{2} + 3 b d - c^{2} = 0.
\]
Thus, the set of $2$-conservative functions in $\nabla \mathcal{P}_3(\R^2, \R)$ is the hypersurface given by the zero locus of $g$. Since this zero locus is not closed under addition, the set of $2$-conservative vector fields in $\nabla\mathcal{P}_3(\R^2, \R)$ is not closed under addition either.

An analogous computation shows that the set of $3$-conservative function is given by the zero locus of 8 homogeneous polynomials of degree 7, each of which is divisible by $g$. Therefore, all $2$-conservative vector fields in $\nabla\mathcal{P}_3(\R^2, \R)$ are also $3$-conservative.

\section{Detailed Proofs}

\subsection{Proof of Theorems \ref{thm:fedavg_convergence_sc} and \ref{thm:heavy_ball_convergence_sc}}\label{appendix:proof_sc}

\begin{proof}
Fix $c \in \{1, \dots, C\}$. By Lemma \ref{lemma:client_properties}, there is some function $h_c \in \mathcal{C}^\infty(\R^n, \R)$ such that $\nabla h_c = I-(I-\gamma\nabla f_c)^k$. Let $\lambda = (\beta-\alpha)/(\beta + \alpha)$. By Assumptions \ref{assm:client_conservative} and \ref{assm:beta_lipschitz}, Lemma \ref{lemma:preserve_properties}, and our assumption on $\gamma$, we find that $h_c$ is $(1-\lambda^k)$-strongly convex and $\nabla h_c$ is $(1+\lambda^k)$-Lipschitz continuous.

Note that the server vector field $V_s$ in \eqref{eq:server_vector_field} is therefore given by $V_s = \nabla f_s$ where
\[
f_s(x) = \dfrac{1}{C}\sum_{c=1}^C h_c(x).
\]
By basic properties of strong convexity and Lipschitz-continuity, we find that $f_s$ is $(1-\lambda^k)$-strongly convex and $\nabla f_s$ is $(1+\lambda^k)$-Lipschitz continuous. In particular, it has a unique minimizer $x_s^*$.

For Theorem \ref{thm:fedavg_convergence_sc}, applying standard results on the convergence of gradient descent on smooth strongly convex functions (in particular, see \cite[Theorem 2.1.15]{nesterov2003introductory}), we find that gradient descent with learning rate $\eta = 1$ on $f_s$ produces iterates $\{x_t\}_{t=0}^\infty$ such that
\[
\|x_{t+1}-x_s^*\| \leq \left(\dfrac{\kappa-1}{\kappa+1}\right)^t\|x_o-x_s^*\|
\]
where $\kappa = (1+\lambda^k)/(1-\lambda^k)$. Some simple algebraic manipulation implies
\[
\dfrac{\kappa-1}{\kappa+1} = \left(\dfrac{\beta-\alpha}{\beta+\alpha}\right)^k
\]
proving \cref{thm:fedavg_convergence_sc}.

For Theorem \ref{thm:heavy_ball_convergence_sc}, we apply standard results on the convergence of gradient descent with heavy-ball momentum (see \cite{polyak1964some}). In particular, by setting the learning rate $\eta$ by
\[
\eta = \dfrac{4}{\left(\sqrt{1+\lambda^k} +\sqrt{1-\lambda^k}\right)^2}
\]
and the momentum parameter $m$ by
\[
m = \max\left\{\left|1-\sqrt{\eta(1-\lambda^k)}\right|, \left|1-\sqrt{\eta(1+\lambda^k)}\right|\right\}^2
\]
we obtain the desired convergence rate.
\end{proof}

\subsection{Proof of Theorem \ref{thm:fedavg_convergence_convex}}\label{appendix:proof_convex}

\begin{proof}
Fix $c \in \{1, \dots, C\}$. By Lemma \ref{lemma:client_properties}, there is some function $h_c \in \mathcal{C}^\infty(\R^n, \R)$ such that $\nabla h_c = I-(I-\gamma\nabla f_c)^k$. By Assumptions \ref{assm:client_conservative} and \ref{assm:beta_lipschitz}, Lemma \ref{lemma:preserve_properties}, and our assumption on $\gamma$, $h_c$ is convex and $1$-Lipschitz continuous. By \cref{lemma:client_properties} and our assumption that $f_c$ has a finite minimizer, $h_c$ has a finite minimizer as well.

Note that the server vector field $V_s$ in \eqref{eq:server_vector_field} is therefore given by $V_s = \nabla f_s$ where
\[
f_s(x) = \dfrac{1}{C}\sum_{c=1}^C h_c(x).
\]
By basic properties of convexity and Lipschitz-continuity, we find that $f_s$ is convex and $1$-Lipschitz continuous. Moreover, the average of convex functions with finite minimizers must also have a finite minimizer, so $f_s$ has some finite minimizer $x_s^*$. By applying standard results on the convergence of gradient descent on smooth convex functions (in particular, see \cite[Theorem 3.3]{bubeck2015convex}), we find that gradient descent with learning rate of $\eta = 1$ on $f_s$ produces iterates $\{x_t\}_{t=0}^\infty$ such that
\[
f_s(x_t)-f_s(x_s^*) \leq \dfrac{1}{2t}\|x_0-x_s^*\|.
\]
\end{proof}

\section{Closed Integral Curves in Federated Learning}\label{appendix:cyclic_fl}

In this appendix we present calculations that demonstrate the possibility of closed integral curves in federated learning with non-convex client losses. The existence of losses of higher regularity than those presented here (e.g. convex or satisfying the PL condition) whose server dynamics admit closed integral curve solutions is an interesting open question. We suspect that examples like this can be transferred to some higher regularity classes, but clearly not all. For example, \cite{charles2021convergence} demonstrate that such integral curves are impossible for quadratic functions (under minor assumptions on learning rates).

Our example dynamics take place in $\R^2$, and we focus on the case of $C = 2$ clients. For $c = 1, 2$ we define a family of functions by

\begin{equation}\label{eq:cyclic_loss_def}
f_c(x, y) := f_c^{(1)}(x, y) + f_c^{(2)}(x, y),
\end{equation}
where
\[
f_c^{(1)}(x, y) := \min\left(\frac{\alpha_c}{2}\left(y - y_c\right)^2 +\frac{\beta_c}{2}\left(x - x_c\right)^2, 1\right), 
\]
\[
f_c^{(2)}(x, y) := 
\min\left(\frac{\alpha_c}{2}\left(y + y_c\right)^2 + \frac{\beta_c}{2}\left(x + x_c\right)^2, 1\right).
\]

We will see that carefully selecting two functions from this family and performing full-gradient \fedavg on these clients will yield server dynamics with closed integral curves. First, note that for any $x_c \text { and } y_c$, $\alpha_c \text{ and } \beta_c$ can be chosen such that the domains of attraction of the terms $f_c^{(1)}$ and $f_c^{(2)}$ are non-overlapping. One can verify that setting $\gamma = 5, \delta = 0.05$, and letting $\alpha_1 = \delta, \beta_1 = \gamma, x_1 = y_1 = 1$, or $\alpha_2 = \gamma, \beta_2 = \delta, x_2 = -1, y_2=1$ satisfies this requirement. Let these choices define the functions $f_1$ and $f_2$.

Now, assume we perform \fedavg with fixed learning rate $\eta > 0$ for some sufficiently large number of local steps $k$. We assume these clients follow full gradient descent, and we choose $k$ large enough so that the clients following full-gradient descent on the losses $f_1$ and $f_2$ converge to a stationary point, independent of the starting point. This can be guaranteed in our setting by setting $k = O\left(\eta^{-1}\right)$, with (easily computable) constant depending on $\gamma$ and $\delta$.

Notice that by assuming clients ``run until convergence'', the form of the server vector field $V_s$ (defined in \eqref{eq:server_vector_field}) becomes quite simple. We define the following domains in the $xy$-plane:

\begin{equation}\label{eq:domain_def}
\begin{split}
\mathbf{I} = \{(x, y) : f_1^{(1)}(x, y) < 1 \}, \\
\mathbf{II} = \{(x, y) : f_1^{(2)}(x, y) < 1 \}, \\
\mathbf{III} = \{(x, y) : f_2^{(1)}(x, y) < 1 \}, \\
\mathbf{IV} = \{(x, y) : f_2^{(2)}(x, y) < 1 \}.
\end{split}
\end{equation}

It is straightforward (though tedious) to verify that our choices of $\gamma, \delta$ above ensure $\mathbf{I} \cap \mathbf{IV}, \mathbf{I} \cap \mathbf{III}, \mathbf{II} \cap \mathbf{III}, \text{ and } \mathbf{II} \cap \mathbf{IV}$ are all nonempty. With these regions defined, a straightforward computation shows that the server vector field $V_s$ is given by:

\begin{equation}\label{eq:cyclic_server_vf}
V_s(x, y) = \begin{cases}
\left(1-x, -y\right) &\quad(x, y) \in \mathbf{I} \cap \mathbf{IV} \\
\left(-x, 1-y\right) &\quad(x, y) \in \mathbf{I} \cap \mathbf{III} \\
\left(-1-x, -y\right) &\quad(x, y) \in \mathbf{II} \cap \mathbf{III} \\
\left(-x, -1-y\right) &\quad(x, y) \in \mathbf{II} \cap \mathbf{IV} \\
\left(1-x, 1-y\right) &\quad(x, y) \in \mathbf{I} \cap \left(\mathbf{III}^c \cup \mathbf{IV}^c\right) \\
\left(-1-x, 1-y\right) &\quad(x, y) \in \mathbf{III} \cap \left(\mathbf{I}^c \cup \mathbf{II}^c\right) \\
\left(-1-x, -1-y\right) &\quad(x, y) \in \mathbf{II} \cap \left(\mathbf{III}^c \cup \mathbf{IV}^c\right) \\
\left(1-x, -1-y\right) &\quad(x, y) \in \mathbf{IV} \cap \left(\mathbf{I}^c \cup \mathbf{II}^c\right) \\
0 &\quad\text{otherwise.} \\
\end{cases}
\end{equation}

We define a flow along this vector field in the usual manner, by the ODE
\[
\frac{d}{dt} \left(x(t), y(t)\right) = V_s(x, y).
\]
That the dynamics of \fedavg will admit closed integral curves in this setting can now be readily seen, either by inspecting \cref{fig:cyclic_server_vf_main} or explicitly following a closed trajectory. The dynamics of \fedavg (as in \eqref{eq:fedavg_update}) correspond to discretizing the ODE above with some step-size $\eta$. That is, \fedavg maps a point $(x_t, y_t)$ to $(x_{t+1}, y_{t+1}) := (x_t, y_t) + \eta V_s(x_t, y_t)$. Under this discretization, letting $(x_0, y_0) = (0, 1)$ and choosing $\eta = 1$ yields a closed trajectory of period 8. Further, the choice of discretization does not affect the nature of the closed curve, only its period, as is clear from \cref{fig:cyclic_server_vf_main}.

\end{document}